\documentclass[12pt]{article}
\usepackage{amsmath}
\usepackage{amsfonts}
\usepackage{amssymb, amsthm}

\newtheorem{theorem}{Theorem}
\newtheorem{prop}{Claim}
\newtheorem*{conj*}{Conjecture}

\theoremstyle{definition}
\newtheorem*{example*}{Example}

\theoremstyle{remark}

\newcommand{\RR}{\mathbb{R}}

\begin{document}
\title{QR-submanifolds and Riemannian metrics with $G_2$ holonomy}
\author{Dmitry Egorov\footnote{This work was supported in part by Russian Foundation for Basic
Research (grant 09-01-00598-a) and the Council of the Russian
Federation Presidential Grants (projects NSh-7256.2010.1 and
MK-842.2011.1).}}
\date{}
\maketitle

\sloppy

\begin{abstract}
In this note we prove that QR-submanifolds of the hyper-K\"{a}hler
manifolds under some conditions admit the $G_2$ holonomy.
We give simplest examples of such QR-submanifolds namely  tori.

We conjecture that all $G_2$ holonomy manifolds arise in this way.
\end{abstract}

\section{Introduction}
The study of $G_2$-manifolds lacks  explicit examples of closed
manifolds. First complete Riemannian metrics with holonomy $G_2$ are
constructed by Bryant and Salamon in \cite{Bryant}. First compact
examples are given by Joyce in \cite{Joyce,Joyce2}. Later Kovalev
constructs more compact examples in \cite{Kovalev,Kovalev2}. Note
that metrics constructed in \cite{Joyce,Joyce2,Kovalev,Kovalev2} are
not explicit.

Lack of examples is a consequence of the fact that $G_2$-manifolds
are not generally algebraic in the broad sense of the term.

In this paper we try to partially explain this fact and conjecture
that $G_2$-manifolds are generally QR-submanifolds of
hyper-K\"{a}hler manifolds. Roughly speaking, QR-submanifolds are
real hypersurfaces of hyper-K\"{a}hler manifolds.

The author is grateful to Iskander Taimanov and Yaroslav Bazaikin
for helpful remarks and support.

\section{Preliminaries}\label{section_pre}
\subsection{$G_2$-structure}\label{section_cross}

Define a $3$-form $\Omega_0$ on $\RR^7$  by
\begin{equation}\label{g2_form}
\Omega_0 = x^{127}  + x^{136}  + x^{145} + x^{235} - x^{246} +
x^{347} + x^{567}.
\end{equation}
By $x^{ijk}$ denote the $x^i\wedge x^j\wedge x^k$. The subgroup of
$GL(7,\RR)$ preserving $\Omega_0$ and orientation is called the
$G_2$ group.

Let $M$ be an oriented closed $7$-manifold. Suppose there exists a
global $3$-form $\Omega$ such that pointwise it coincides with
$\Omega_0$; then $M$ is called a $G_2$-manifold or we say that $M$
carries the $G_2$-structure. It is known that the orientation and
the Riemannian metric are uniquely determined by the
$G_2$-structure.

\subsection{Cross products}
Let $M$ be a $G_2$-manifold. Suppose a multilinear alternating
smooth map $P:TM\times TM\to TM$. Suppose $P$ satisfies
compatibility conditions:
\begin{equation}\label{cross_def1}
g(P(e_1, e_2),e_i) = 0,\quad  i = 1,2;
\end{equation}
\begin{equation}\label{cross_def2}
\|P(e_1, e_2)\|^2 = \|e_1\|^2\|e_2\|^2 - g(e_1,e_2)^2,\quad \|e\|^2
= g(e,e).
\end{equation}
Then $P$ is called a  cross product. We also denote $P(e_1,e_2)$ by
$e_1\times e_2$.

The cross product is uniquely determined by the $3$-form $\Omega$:
\begin{equation}\label{form2prod} \Omega(e_1,e_2,e_3) = g(
P(e_1, e_2),e_3).
\end{equation}
Conversely, the cross product defines  the metric by the following
formula:
\begin{equation}\label{cross_induces_J}
P(e_1,P(e_1, e_2)) = -\|e_1\|^2e_2 + g(e_1,e_2)e_1.
\end{equation}
Using \eqref{form2prod}, we determine the $3$-form $\Omega$ from the
cross product and the metric. Thus the cross product implies the
$G_2$-structure and vice versa.

Recall that if cross product is parallel with respect to the metric
connection, then the holonomy group of $M$ is a subgroup of $G_2$
and coincides with $G_2$ iff $\pi_1(M)$ is a finite group
\cite{Joyce}.

\subsection{QR-submanifolds}
Riemannian $4n$-manifold with holonomy group contained in $Sp(n)$ is
called a hyper-K\"{a}hler manifold.

Suppose  $M$ is a submanifold of the hyper-K\"{a}hler $\overline{M}$
such that normal bundle of $M$ is the direct sum of $\nu$ and
$\nu^\perp$ and
\begin{equation}\label{qr} J_i\nu\subset \nu,\quad
J_i\nu^\perp\subset TM,\quad  i = 1,2,3,
\end{equation}
where by $J_i$ we denote the ith complex structure of
$\overline{M}$. Then $M$ is called a QR-submanifold of
$\overline{M}$.

In what follows we consider QR-submanifolds with $\dim\nu^\perp=1$
only. We  call them QR-submanifolds of the hypersurface type.

\section{The main result}
\label{section_main}

\begin{theorem}
Let  $M$ be an oriented  $7$-manifold. If $M$ is a hypersurface type
QR-submanifold of hyper-K\"{a}hler $\overline{M}$, then there exists
the $G_2$-structure on $M$.
\end{theorem}
\begin{proof}
We shall construct a cross product on $M$ such that it is compatible
with the induced metric.

By \eqref{qr}, it follows that $\xi_i = J_in$ are $3$ non-vanishing
vector fields on $M$. This agrees with \cite{EmeryThomas}, where
existence of two non-vanishing vector fields on arbitrary compact
orientable $7$-manifold was shown. Third non-vanishing vector is the
cross product of the first two (see also \cite{Akbulut}).

We may assume that $\xi_i$ are unit orthogonal with respect to the
induced metric vector fields on $M$. Locally we extend $\xi_i$ to a
basis. Additional vectors are denoted by $\xi_\alpha$, i.e., by
Greek indices.

Let the cross product $P$ be given by the following formulae:
\begin{eqnarray}
\label{P1}&& P(\xi_i,\xi_j) = \xi_k,\quad (ijk)\in(123),\\
\label{P2}&& P(\xi_i,\xi_\alpha) = J_i(\xi_\alpha),\\
\label{P3}&& P(\xi_\alpha,J_i(\xi_\alpha)) = \xi_i.
\end{eqnarray}
By the definition of a hypersurface type QR-submanifold, we have
that for any $\xi_\alpha$, $\xi_\beta$ there exists complex
structure $J_i$ such that $J_i\xi_\alpha = \xi_\beta$. Hence
formulae \eqref{P1}--\eqref{P3} define the cross product on all
basis vectors.

Clearly, $P$ satisfies  \eqref{cross_induces_J} and therefore $P$ is
compatible with the induced metric.
\end{proof}

Let's find out when the constructed cross product is parallel that
is when holonomy is reduced to a subgroup of $G_2$.

Let $\overline{\nabla}$ and $\nabla$ be a metric connection on
$\overline{M}$ and $M$ respectively.

\begin{prop}
\begin{eqnarray}
\label{xi3} \nabla\xi_i &=&
J_i(\overline{\nabla}n) - b(\xi_i).\\
\label{J} (\nabla J_i)(\xi_\alpha) &=& J_i\circ b(\xi_\alpha) -
b\circ J_i(\xi_\alpha).
\end{eqnarray}
\end{prop}
\begin{proof}
By the Gauss formula,  we have
\begin{equation}\label{xi1}\overline{\nabla}\xi_i
= \nabla\xi_i + b(\xi_i),\end{equation} where $b(\xi_i) = b(\xi_i,
\cdot)$ and  $b$ is the second fundamental form.

Also, the definition of the hyper-K\"{a}hler manifold implies that
\begin{equation}\label{xi2}
\overline{\nabla}\xi_i = \overline{\nabla}J_i(n) =
(\overline{\nabla}J_i)(n) + J_i(\overline{\nabla}n) =
 J_i(\overline{\nabla}n).
\end{equation}
Combining \eqref{xi1} and \eqref{xi2}, we get \eqref{xi3}.

Similarly, combining
\begin{equation}\label{J1}
\overline{\nabla}(J_i\xi_\alpha) = \nabla(J_i(\xi_\alpha)) +
b(J_i(\xi_\alpha)) = (\nabla J_i)(\xi_\alpha) + J_i(\nabla
\xi_\alpha) + b(J_i(\xi_\alpha))
\end{equation}
and
\begin{equation}\label{J2}
 \overline{\nabla}(J_i\xi_\alpha) = (\overline{\nabla}J_i)(\xi_\alpha)
+ J_i(\overline{\nabla}\xi_\alpha) =
J_i(\overline{\nabla}\xi_\alpha),
\end{equation}
we have \eqref{J}.
\end{proof}

By definition, put
$$
X_i(\xi) = J_i(\overline{\nabla}_\xi n) - b(J_in,\xi),\quad
Y_i(\xi,\eta) = J_i b(\xi,\eta) - b( \xi, J_i\eta).
$$

\begin{prop}
\begin{eqnarray}
\label{calc1} (\nabla P)(\xi_i,\xi_j) &=& X_k - X_i\times\xi_j -
\xi_i\times X_j.\\
\label{calc2} (\nabla P)(\xi_i,\xi_\alpha) &=& Y_i(\xi_\alpha) -
X_i\times\xi_\alpha.\\
\label{calc3}(\nabla P)(\xi_i,\xi_\alpha) &=& Y_i(\xi_\alpha) -
X_i\times\xi_\alpha.
\end{eqnarray}
\end{prop}
\begin{proof}
Let's prove \eqref{calc1}. We differentiate \eqref{P1}:
\begin{equation}\label{calc11}
(\nabla P)(\xi_i,\xi_j) =  \nabla\xi_k - P(\nabla\xi_i,\xi_j) -
 P(\xi_i,\nabla\xi_j).
\end{equation}
Combining \eqref{xi3}, \eqref{calc11} and \eqref{P1}, we obtain
\eqref{calc1}.

Similarly, if we differentiate \eqref{P2} and \eqref{P3}, we get
\eqref{calc2} and \eqref{calc3}.

\end{proof}

Recall that $\nabla P=0$ implies that $\mathrm{Hol}(M)\subset G_2$.
If we equate with zero formulae \eqref{calc1}--\eqref{calc3}, then
we obtain sufficient conditions for $\nabla P=0$. Note that
\eqref{calc2} and \eqref{calc3} are equivalent.

\begin{theorem}\label{thm2}
Suppose  $M$ is an oriented  $7$-manifold such that  $M$ is a
hypersurface type QR-submanifold of the hyper-K\"{a}hler
$\overline{M}$.
If the following equations
hold:
\begin{equation}\label{final1}
X_k(\xi) - X_i(\xi)\times \xi_j - \xi_i\times X_j(\xi)=0,
\end{equation}
\begin{equation}\label{final2}
Y_i(\xi,\eta) - X_i(\xi)\times\eta = 0,
\end{equation}
for any $\xi,\eta,J_i\eta\in\Gamma(TM)$, $i=1,2,3$; then holonomy
group of $M$ is contained in $G_2$.
\end{theorem}

\begin{example*}
Simplest examples of QR-submanifolds with holonomy contained in
$G_2$ are totally geodesic hypersurfaces. These are  flat tori:
$T^7\hookrightarrow T^8$ and $T^3\times K3\hookrightarrow T^4\times
K3$.
\end{example*}

\section{Conjecture}\label{section_conj}
Emery Thomas proves in \cite{EmeryThomas} that any $G_2$-manifold
admits $3$ non-vanishing unit vector fields $\xi_i$. There exists a
complex structure on $\xi_i^\perp$ determined by
\eqref{cross_induces_J}.
Verbitsky shows in  \cite{Verb} that these complex structures are
integrable iff the holonomy is contained  in $G_2$. Due to
integrability we formulate the following
\begin{conj*}
Any $G_2$ holonomy  manifold is a QR-submanifold  of a certain
hyper-K\"{a}hler manifold satisfying the conditions of Theorem
\ref{thm2}.
\end{conj*}

\vskip5mm {\sc Ammosov Northeastern Federal University, Yakutsk,
677000,
Russia}\\
\textit{e-mail:}{\tt egorov.dima@gmail.com}

\end{document}